\documentclass[reqno,11pt]{amsart}
\newtheorem{theorem}{Theorem}[section]
\newtheorem{lemma}[theorem]{Lemma}
\newtheorem {proposition}[theorem]{Proposition}
\theoremstyle{definition}
\newtheorem{definition}[theorem]{Definition}

\theoremstyle{remark}
\newtheorem{remark}[theorem]{Remark}

\numberwithin{equation}{section}

\usepackage{geometry}
\usepackage{mathpazo} \usepackage[pdftex]{hyperref}
\newcommand{\norm}[1]{\left\Vert#1\right\Vert}
\newcommand{\scal}[1]{\left<#1\right>}

\newcommand{\C}{\mathbb C}
\newcommand{\R}{\mathbb R}
\newcommand{\Z}{\mathbb Z} 	

\newcommand{\V}{\R^d}

\newcommand{\Fnugr}{\mathcal{F}^{2,\nu}_{\Gamma, \chi}}
\newcommand{\Fnugrt}{\mathcal{F}^{2,\nu}_{\widetilde{\Gamma}, \widetilde{\chi}}}

\newcommand{\normqr}[1]{\left\Vert#1\right\Vert_{\Gamma,\nu}}

\newcommand{\scalqr}[1]{\left<#1\right>_{\Gamma,\nu}}
\newcommand{\normnu}[1]{\left\Vert#1\right\Vert_{\Gamma,H}}
\newcommand{\scalnue}[1]{\left<#1\right>_{\Gamma,H}}
\newcommand{\mes}{d\lambda(x)}
\newcommand{\mess}{d\lambda}
\newcommand{\mesw}{d\lambda(x_1)}
\newcommand{\meswo}{d\lambda(x_2)}
\newcommand{\mesc}{d\lambda(z)}
\newcommand{\meswc}{d\lambda(z_1)}
\newcommand{\meswoc}{d\lambda(z_2)}

\newcommand{\LQ}{\mathcal{C}^{\nu,\infty}_{\Gamma,\chi}(\V)}
\newcommand{\LQc}{L^{2,\nu}_{\Gamma,\chi}(\V)}
\newcommand{\WR}{\mathbb{V}_\Gamma}
\newcommand{\WC}{\mathbb{V}_\C}
\newcommand{\WCo}{{\mathbb{V}_\C}^{\perp_H}}
\newcommand{\WRo}{{\mathbb{V}_\Gamma}^{\perp_{\scal{.,.}}}}

\newcommand{\Vc}{\C^d}

\begin{document}

\title[Likewise theta functions of rank $r$ on $\R^d$]{Likewise theta functions of rank $r$ on $\R^d$: analytic properties and associated Segal-Bargmann transform}

\author{A. Ghanmi}     
\author{A. Intissar}   
\author{Z. Mouhcine}   
\author{M. Ziyat}      
\address{A.G.S.-L.A.M.A, 
          Department of Mathematics, P.O. Box 1014,  Faculty of Sciences,
          Mohammed V University in Rabat, Morocco}

\thanks{A. Ghanmi, A. Intissar and Z. Mouhcine are partially supported by CNRST-URAC/03 and
the Hassan II Academy of Sciences and Technology, Morocco. M. Ziyat is partially supported by the CNRST grant 56UM5R2015, Morocco.}



\begin{abstract}
We introduce and study the Hilbert space of $(L^2,\Gamma,\chi)$-likewise theta functions on $\V$ with respect to a given discrete subgroup $\Gamma$ of arbitrary rank and a character $\chi$ of $\Gamma$. A concrete description is given and an orthonormal basis is then constructed.  Its range by the classical Segal-Bargmann transform is also characterized and leads to the so-called theta-Bargmann Fock space.
\end{abstract}

\maketitle

\section{Introduction}
Analytic properties of the holomorphic automorphic functions, associated with a full-rank lattice $\Gamma$ in the $d$-dimensional complex space $\C^d$ and a given mapping $\chi$ on $\Gamma$ with values in unit circle of $\C$, are well studied in the literature. 
Such functions play important roles in number theory and abelian varieties \cite{Igusa1972,Serre1973,BumpFriedbergHoffstein1996,PolishChuk2003}, cryptography and coding theory \cite{Ritzenthaler2004,ShaskaWijesiri2008}, as well as in quantum field theory \cite{Cartier1966}.
Extending these properties to the case of an arbitrary rank discrete subgroup is an interesting area of research. The more recent investigation in this context has been discussed in \cite{GhIn2012,Souid2015} for rank one discrete subgroups and next generalized in \cite{GhInSou2016} to isotropic discrete subgroups of rank less or equal to $d$. The elaboration of these properties lies in the holomorphic character of the considered functions attached to the complex structure of $\C^d$. This tool is lost when working on $\R^d$ instead of $\R^{2d}=\C^d$, where $d$ is not necessary even.
Thus, inspired by the impact of Segal-Bargmann transform on signal processing and time-frequency analysis on the free Hilbert space $L^{2}(\R^d)$
(see for example \cite{grocheniq}) and  motivated by the fact that many signals in practice are quasi-periodic, we will develop and investigate in a natural way a parallel theory for the space $\LQc$ of $(L^2,\Gamma,\chi)$-likewise theta functions associated to an arbitrary discrete subgroup of rank $r$ in $\V$.
For the full-rank lattice $\Gamma$, the Segal-Bargmann transform of $(L^2,\Gamma)$-periodic functions is characterized to be the space of $L^2$-Bloch wave functions \cite{BouchibaIntissar2001}.

The aim of the present paper is then two folds. Firstly, we give explicit description of the elements of $\LQc$ and next construct an explicit orthonormal basis in terms of the modified Fourier expansion and Hermite polynomials.
Secondly, we consider the Segal-Bargmann transform and prove that it maps isometrically the space of $(L^2,\Gamma,\chi)$-likewise theta functions on $\R^d$ onto the  well-studied space of $(L^2,\widetilde{\Gamma},\widetilde{\chi})$-holomorphic theta functions on $\Vc$ considered in \cite{GhInSou2016} for a special pair $(\widetilde{\Gamma};\widetilde{\chi})$. This gives rise to the introduction of the so-called  theta-Segal-Bargmann transform involving an integral representation over a fundamental domain with kernel function given in terms of the multidimensional Reimann theta function with special characteristics.

This paper is organized as follows. Section 2 is devoted to the exact statement of our main results.
In Section 3, we establish some basic properties of the space $\LQc$ of the $(L^2,\Gamma,\chi)$-likewise theta functions on $\R^d$ and give the proof of the Theorem \ref{mainresult1}.
In Section 4, we review the needed properties concerning the space $\Fnugrt(\Vc) $ of $(L^2,\widetilde{\Gamma},\widetilde{\chi})$-holomorphic theta functions considered in \cite{GhInSou2016}.  In Section 5, we prove Theorem \ref{mainresult2} concerning to the characterization of the range of $\LQc$ by the Segal-Bargmann transform as well as Theorem \ref{mainresult3} leading to the notion of theta-Segal-Bargmann transform.

\section{Statement of main results}
In order to give a concise picture of our results, we endow the $d$-dimensional Euclidean space $\V$
with the usual scalar product $\scal{\cdot,\cdot}$ and norm $\|\cdot\|$, and let $\Gamma=\Gamma_r$ be a discrete subgroup of rank $r$; $ r=0,1,\cdots,d$, in the additive group $(\V,+)$. By $\V/\Gamma$ we denote the associated abelian orbital group equipped with the quotient topology and the Haar measure.
Associated with the data of $\Gamma$, a nonnegative real number $\nu$ and a given mapping $\chi$ with values in the unit circle of $\C$ on $\Gamma$, we consider the functional space $\LQc$ of Borel measurable functions $f$ on $\V$ satisfying the functional equation
\begin{equation}\label{FctEq2}
	f(x+\gamma) = \chi(\gamma)  e^{\nu \scal{x+\frac{\gamma}2, \gamma} } f(x),
\end{equation}
for almost every $x\in\V$ and every $\gamma\in\Gamma$, and
\begin{equation}\label{Intrnorm}
	\normqr{f}^2 := \int_{\Lambda(\Gamma)} |f(x)|^2  e^{-\nu \norm{x}^2} \mes<\infty.
\end{equation}
Here $\Lambda(\Gamma)$ is a fundamental domain of $\Gamma$ in $\V$ representing $\V/\Gamma$ and gives rise to a compact fundamental domain of $\Gamma$ in $\WR=\mbox{Span}_\R(\Gamma)$, the $r$-dimensional real vector space generated by $\Gamma$.
Notice that the quantity $\normqr{\cdot}$ makes sense and it is independent of the choice of the fundamental domain, for the function $x\longmapsto|f(x)|^2e^{-\nu\norm{x}^2}$ being a $\Gamma$-periodic on $\V$ for every given $f$ satisfying \eqref{FctEq2}.
Furthermore, $\normqr{\cdot}$ defines a norm on $\LQc$. The corresponding scalar product is given by
\begin{equation}\label{IntrScal1}
\scalqr{f,g} := \int_{\Lambda(\Gamma)} f(x) \overline{g(x)}  e^{-\nu \norm{x}^2} \mes .
\end{equation}
We will show that the functional space $\LQc$ is nontrivial if and only if $\chi$ is a character. In this case $\chi$ has a representation of the form $\chi(\gamma)=e^{2i\pi\scal{\gamma,v_\chi}}$ (see Lemma \ref{Lem:condidatFct}).
Under such condition, we give a concrete description of $\LQc$ using Fourier analysis related to the dual lattice $\Gamma^*$ of $\Gamma$. Namely,
by identifying $\V$ to $\WR \times \WRo$, where $\WRo$ is the complement orthogonal of $\WR$ in $\V$ with respect to $\scal{.,.}$, and denoting
by $ \mathbf{H}^\nu_{\mathbf{k}}$ the $(d-r)$-dimensional Hermite polynomials, we can assert the following.

	\begin{theorem} \label{mainresult1}
 The family of functions on $\WR \times \WRo$ defined by
		\begin{equation}
		e_{{\gamma^*},\mathbf{k}}(x)=e_{{\gamma^*},\mathbf{k}}(x_1,x_2) := e^{\frac{\nu}2 \scal{x_1,x_1}+ 2\pi i \scal{v_\chi+{\gamma^*},x_1} }
          \mathbf{H}^\nu_{\mathbf{k}}(x_2)    
		\end{equation}
		for varying ${\gamma^*}\in \Gamma^*$ of $\Gamma$ and  $\mathbf{k}=(k_{1},  \cdots ,  k_{d-r})\in(\Z^+)^{d-r}$,
        constitutes an orthogonal basis of $\LQc$ with
        \begin{equation}			
		\normqr{ e_{{\gamma^*},\mathbf{k}}}^2=vol(\Lambda_1(\Gamma))\left(\frac{\pi}{\nu}\right)^{(d-r)/2}	 2^{|\mathbf{k}|}\mathbf{k}!.
		\end{equation}
	\end{theorem}

The characterization of the range of $\LQc$ by the Segal-Bargmann transform,
\begin{align}\label{bargmannint}
		[\mathcal{B}\varphi](z) =\left(\frac{\nu}{\pi}\right)^{\frac{3d}{4}}
		\int_{\V} e^{\sqrt{2}\nu \scal{z,x} -\frac{\nu}{2}\scal{z,z}} \varphi(x)e^{-\nu \norm{x}^2}\mes; \quad z\in \Vc,
		\end{align}
involves the shifted lattice $\widetilde{\Gamma}:=\Gamma/{\sqrt{2}}$ and the $\widetilde{\Gamma}$-character defined by
$\widetilde{\chi}(\widetilde{\gamma})=e^{2i\pi\scal{\widetilde{\gamma},\sqrt{2}v_\chi}}. $
The natural extension of $\scal{\cdot,\cdot}$ on $\V$ to $\Vc=\WC \oplus \WCo$ with $\WC=\WR+i\WR$ is also denoted by $\scal{\cdot,\cdot}$.

		\begin{theorem}\label{mainresult2}
        The Segal-Bargmann transform $\mathcal{B}$ given through \eqref{bargmannint} defines an isometric isomorphism form $\LQc$ onto $\Fnugrt(\Vc)$.
		\end{theorem}

The last theorem of this paper gives another integral representation for $\mathcal{B}$ when restricted to $\LQc$. This representation is a coherent state transform and encodes the Hilbert structures of $\LQc$ and $\Fnugrt(\Vc)$. In fact, the kernel function is the bilateral generating function making appeal of both bases. Moreover, it can be expressed in terms of the multidimensional Riemann theta function $\Theta_{\alpha,\beta} (z \big |  \Omega)$ with special characteristics. To this end, recall that (\cite{Riemann1857,Mumford83}):
	\begin{align}\label{defTheta}
	\Theta_{\alpha,\beta} (z \big |  \Omega)
	= \sum\limits_{n\in\Z^r} 	e^{2i\pi \left\{ \frac 12 (\alpha+n)\Omega(\alpha+n) + (\alpha+n)(z+\beta)\right\} }
	\end{align}
where $\alpha,\beta\in \R^r$ and $\Omega$ a symmetric matrix in $\C^{r\times r}$ with strictly positive definite imaginary part.
The positive definiteness of $\Im(\Omega)$ guarantees the convergence of \eqref{defTheta} on $\C^{r}$.

		\begin{theorem}\label{mainresult3}  The Segal-Bargmann transform $\mathcal{B}$ on $\LQc$ is also given by
			\begin{equation}
			[\mathcal{B}\varphi](z)=  \int_{\Lambda(\Gamma)} A^\nu_{\Gamma,\chi}(z;x) \varphi(x)e^{-\nu \|x\|^2} \mes; \quad z\in\Vc,
			\end{equation}
			where the kernel function $A^\nu_{\Gamma,\chi}(z;x)$ is given in terms of the modified theta function as follows
			\begin{equation}
			A^\nu_{\Gamma,\chi}(z;x)=\left(\frac{\nu}{\pi}\right)^{\frac{3d}{4}}e^{\sqrt{2}\nu \scal{z,x}-\frac{\nu}{2}\scal{z,z}}
            \Theta_{0,G\beta_\chi}\left(\frac{i\nu}{2\pi}G(x_1-\sqrt{2}z_1)\bigg|\frac{i\nu}{2\pi}G\right).
			\end{equation}
			Here $x=x_1+x_2\in \WR \oplus \WRo$, $z=z_1+z_2\in \WC \oplus \WCo$, $G:= \left(\scal{\omega_i,\omega_j}\right)_{1\leq i,j\leq r}$ is the Gram-Schmidt matrix of the form $\scal{\cdot,\cdot}$ restricted to the vector space $\WR$ generated by a basis $\{\omega_i; i=1,2, \cdots,r\}$ of $\Gamma$, and $\beta_\chi=(\beta_1,\cdots,\beta_r)\in\R^r$ are the coordinates of $v_\chi$.
		\end{theorem}

\section{Basic properties of $\LQc$ and proof of Theorem \ref{mainresult1} }

Keep notations as above and provide the following definition.

\begin{definition}
	The space $\LQc$ will be called the space of $(L^2,\Gamma,\chi)$-likewise theta functions on $\V$ attached to the discrete subgroup $\Gamma$.
    Similarly, we define $\LQ$ to be the space of $(\mathcal{C}^\infty,\Gamma,\chi)$-likewise theta functions, that is the space of all $\mathcal{C}^\infty$-complex-valued functions $f$ on $\V$ satisfying
	\begin{equation}\label{FctEq1}
		f(x+\gamma) = \chi(\gamma)  e^{ \nu\scal{x+\frac{\gamma}2, \gamma} } f(x),
	\end{equation}
	for every $ x\in\V$ and $\gamma\in\Gamma$. 	
\end{definition}

The existence of such spaces is encoded in the data $(\nu,\Gamma,\chi)$. In fact, the following result gives
a necessary and sufficient condition to $\LQc$ and $\LQ$ be nontrivial.

\begin{proposition}
The vector space  $\LQc$   (resp. $\LQ$) is nonzero if and only if $\chi$ is a character on
$\Gamma$, i.e., for all $\gamma,\gamma'\in\Gamma$, we have
\begin{equation}\label{caracter}
\chi(\gamma+\gamma')=\chi(\gamma)\chi(\gamma').
\end{equation}
\end{proposition}

\begin{proof}
For the necessary condition with $\LQ$, assume that $\LQ$ is nontrivial and let $f$ be a nonzero
function belonging to $\LQ$. Hence, for every $\gamma,\gamma'\in\Gamma$ and $x\in\V$, we can write $f(x+\gamma+\gamma')$
in the following forms
\begin{align}
f(x+\gamma+\gamma') = f((x+\gamma)+\gamma') 
=\chi(\gamma')\chi(\gamma)e^{\nu\scal{x+\frac{\gamma+\gamma'}{2},\gamma+\gamma'}}f(x) \label{eq1}
\end{align}
and
\begin{align}
f(x+\gamma+\gamma') = f(x+(\gamma+\gamma'))=\chi(\gamma'+\gamma)e^{\nu\scal{x+\frac{\gamma+\gamma'}{2},\gamma+\gamma'}}f(x). \label{eq2}
\end{align}
Hence the desired result follows by equating the right hand-sides of \eqref{eq1} and
 \eqref{eq2} and using the fact that $f(x_0)\ne 0$ for certain $x_0$.
The necessary condition concerning $\LQc$ can be handled in a similar way with additional consideration. In fact, the $x_0$ is taken in $\V\setminus\widetilde{\mathcal{D}_f}$, where $\widetilde{\mathcal{D}_f}=\cup_{\gamma\in\Gamma}(D_f+\gamma)$ and $D_f$ is a negligible subset of $\V$ such that $f$ is well defined on its complementary $\V\setminus D_f$. Notice that $\widetilde{\mathcal{D}_f}$ is also negligible since $\Gamma$ is a countable set.

The proof of the sufficient condition follows by considering the Poincar\'{e} series
\begin{align}\label{PSeries}
[{\mathcal{P}}_{\Gamma,\chi}\psi](x)=\sum\limits_{\gamma\in\Gamma} \overline{\chi(\gamma)}e^{-\nu\scal{x+\frac{\gamma}{2},\gamma}}\psi(x+\gamma)
\end{align}
related to given $\mathcal{C}^\infty$ complex-valued function $\psi$ with compact support contained in the interior of $\Lambda(\Gamma)$ (such $\psi$ exists
from the classical analysis). The series \eqref{PSeries} is then well defined, since every $x\in\V$ has a unique representation in $\Lambda(\Gamma)$.
 Moreover, ${\mathcal{P}}_{\Gamma,\chi}\psi$ is a nonzero function belonging to $\LQ \cap \LQc$. Indeed, we have
$ [{\mathcal{P}}_{\Gamma,\chi}\psi]=\psi  \ne 0 $ on $\Lambda(\Gamma)$, since $\chi(0)=1$ and $x+\gamma \notin \mbox{Supp}(\psi)$ for all $x\in \Lambda(\Gamma)$. Now, making use of the fact that
$\chi(\gamma"-\gamma) =\chi(\gamma")\overline{\chi(\gamma)},$
for $\chi$ being a character, combined with 
the symmetry of the bilinear form $\scal{\cdot,\cdot}$, we obtain
\begin{align*}
[{\mathcal{P}}_{\Gamma,\chi}\psi](x+\gamma) = \chi(\gamma)e^{\nu\scal{x+\frac{\gamma}{2},\gamma}} [{\mathcal{P}}_{\Gamma,\chi}\psi](x)
\end{align*}
for every $\gamma \in \Gamma$ and $x\in \V$.  This completes the proof.
\end{proof}

\begin{remark}
The proof of "if" in the previous theorem can be reworded if we are able to exhibit an explicit nonzero complex-valued function belonging to $\LQ\cap \LQc$.
This is contained in Lemma \ref{Lem:condidatFct} below.
\end{remark}

\begin{remark}
Notice that $\LQ$ is a prehilbertian space when equipped with the scalar product \eqref{IntrScal1}. Thus, by functional analysis theory it has a completion.
It should be proved later that $\LQc$ is in fact the completion of $\LQ$ (Proposition \ref{prop:completion}).
\end{remark}

Now, we deal with $\LQc$ under the assumption that $\chi$ is a character. The discrete subgroup $\Gamma$ of $(\V,+)$ can be viewed as a $\Z$-module of dimension $r$. Therefore, $\Gamma$ take the following form
$$\Gamma=\Z\omega_1+\cdots +\Z\omega_r$$
for some $\R$-linearly independent vectors $\omega_1,\cdots,\omega_r\in\V$.
A corresponding fundamental domain is proved to be given by
$$
\Lambda(\Gamma) \simeq \V/\Gamma  \simeq  \Lambda_1(\Gamma)\times \WRo,
$$
where $\Lambda_1(\Gamma)$ is a compact fundamental domain of $\Gamma$ in the $r$-dimensional real vector space $\WR=\mbox{Span}_\R(\Gamma)$, constituting of all $\R$-linearly finite combinations of elements of $\Gamma$. By $\WRo$ we denote its complement orthogonal with respect to $\scal{.,.}$. Accordingly, we can split $\V$ as
\begin{equation}\label{splitV}
\V= \WR \oplus \WRo
\end{equation}
and the symmetric bilinear form $\scal{\cdot,\cdot}$ as
\begin{equation}\label{splitQ}
\scal{x,y}= \scal{x_1,y_1}+\scal{x_2,y_2}=Q_1(x_1,y_1)+Q_2(x_2,y_2),
\end{equation}
for every $x=x_1+x_2$ and $y=y_1+y_2$ in $\V$ with $x_1,y_1\in\WR$ and $x_2,y_2\in\WRo$,
where $Q_1$ (resp. $Q_2$) is the restriction of $\scal{.,.}$ to $\WR$ (resp. $\WRo$) and can be viewed as a positive definite symmetric bilinear form on $\WR$ (resp. $\WRo$).

Now in view of \eqref{splitQ}, the functional equation \eqref{FctEq2} reduces further to
\begin{equation}\label{func-eq}
f(x_1+\gamma, x_2) =\chi(\gamma) e^{\nu \scal{x_1+\frac{\gamma}2, \gamma} } f(x_1, x_2)
\end{equation}
with $f(x_1, x_2):= f(x_1+x_2)$ for all $x_1\in\WR$, $x_2\in\WRo$ and $\gamma\in\Gamma$.

\begin{lemma} \label{Lem:condidatFct}
Under the assumption that $\chi$ is a character, we have the following
\begin{itemize}
\item[i)] There exists a vector $v_\chi\in \WR$ (depending in the choice of the basis of $\Gamma$) such that for all $\gamma\in\Gamma$, we have
    $$\chi(\gamma)= e^{2i\pi \scal{v_\chi,\gamma}} .$$ 
\item[ii)] The function
   \begin{equation}\label{groundfct}
		x \longmapsto   \psi_{v_\chi}(x) : = e^{\frac{\nu}{2}\scal{x_1,x_1} + 2i\pi \scal{x_1,v_\chi}}
   \end{equation}
belongs to $\LQ \cap \LQc$.
\end{itemize}
\end{lemma}

\begin{proof}	
Since $\chi$ is a $\Gamma$-character and its restriction to any $\Z \omega_j$; $1\leq j\leq r$, is a $\Z$-character, there exists $\alpha_j \in \R$ such that $\chi(\omega_j)=e^{2i\pi\alpha_j}$. Therefore,
$$\chi(\gamma) = \chi(m_1\omega_1+\cdots+m_r\omega_r) = e^{2i\pi(\alpha_1m_1+\cdots+\alpha_rm_r)}.$$
Let $\alpha_\chi$ stands for $\alpha_\chi=(\alpha_1, \cdots ,\alpha_r)\in \R^r$ and $m=(m_1,\cdots,m_r)\in \Z^r$. Then, we can rewrite $\alpha_1m_1+\cdots+\alpha_rm_r$ as
$$\alpha_1m_1+\cdots+\alpha_rm_r =
{^t}\hspace*{-.05cm} m G (G^{-1}\alpha_\chi)=
{^t}\hspace*{-.05cm} m G \beta_\chi = \scal{v_\chi,\gamma},
$$
where $G:= \left(\scal{\omega_i,\omega_j}\right)_{1\leq i,j\leq r}$ is the invertible Gram-Schmidt matrix of $Q_1$ on the vector space $\WR$ and $v_\chi$ is the vector in $\WR$ given by
$v_\chi  = \beta_1\omega_1 + \cdots +\beta_r\omega_r$
with
$\beta_\chi :=(\beta_1, \cdots ,\beta_r) := G^{-1} \alpha_\chi$.
This completes the proof of i).

The proof of ii) can be handled easily using the bilinearity and the symmetry of $\scal{.,.}$, keeping in mind that $\chi(\gamma)= e^{2i\pi \scal{v_\chi,\gamma}}$. Indeed, we have
\begin{align*}
\psi_{v_\chi}(x+\gamma) = e^{\frac{\nu}{2} \scal{x_1+\gamma, x_1+\gamma} + 2i \pi \scal{x_1+\gamma, v_\chi}}
= \chi(\gamma) \psi_{v_\chi}(x) e^{\nu \scal{x + \frac{\gamma}{2}, \gamma}}.
\end{align*}
The last equality follows since $\scal{y_1 , \gamma} =  \scal{y , \gamma}$ for every $y=y_1+y_2 \in \V$ with $(y_1,y_2)\in\WR\times\WRo$ and $\gamma \in \Gamma \subset \WR$.
\end{proof}		

The existence and the explicit expression of $\psi_{v_\chi}$ will play a crucial role in establishing the following result.
			
\begin{theorem}\label{thm-description} Keep notations as above. Then, $\LQc$ is a Hilbert space. Moreover, a
 function $f$ belongs to $\LQc$ if and only if it can be expanded in series as
			\begin{equation}\label{expansion0}
			f(x)=e^{\frac{\nu}2\scal{x_1,x_1}} \sum\limits_{{\gamma^*}\in\Gamma^*} a_{{\gamma^*}}(x_2)
            e^{2i\pi \scal{v_\chi+{\gamma^*},x_1}},
			\end{equation}
      where $\Gamma^* := \{ \gamma^* \in \WR ; \, \scal{\gamma^*, \gamma} \in  \Z; \, \, \gamma\in \Gamma\}$ denotes the dual of $\Gamma$.
       The involved coefficients $a_{{\gamma^*}} (x_2)$; $x_2\in \WRo$, satisfy the growth condition
      	\begin{equation}\label{condphi}
      	\sum\limits_{{\gamma^*}\in\Gamma^*}\| a_{\gamma^*}\|^2_{L^2(\WRo, e^{-\nu \scal{x_2 ,x_2}} \meswo)}<+\infty.
      	\end{equation}
\end{theorem}

      \begin{proof} The first assertion is obvious. Indeed, for any Cauchy sequence $(f_n)_n$ in $\LQc$, the $({f_n}_{|_{\Lambda(\Gamma)}})_n$ is a Cauchy sequence in the Hilbert space $L^2(\Lambda(\Gamma);e^{-\norm{\xi}^2}d\lambda)$ and hence converges to some function $f$ defined on $\Lambda(\Gamma)$. Since $\Lambda(\Gamma)$ is an arbitrary fundamental domain, it follows that $f$ is defined on $\V$ and satisfies \eqref{FctEq2}.
For the second assertion, let $f\in \LQc$. Then, the function
$$\varphi(x) := e^{-\frac{\nu}{2}\scal{x_1,x_1}-2i\pi \scal{v_\chi,x_1}}f(x)  $$
  is a $\Gamma$-periodic function on $\V$ in the $x_1$-direction. Furthermore, by means of Fubini's theorem, we get
\begin{align}\label{Fubinin}
\normqr{f}^2  &= \int_{\WRo}
\left(  \int_{\Lambda_1(\Gamma)} \left|\varphi(x_1,x_2)\right|^2 \mesw\right)e^{-\nu \scal{x_2,x_2}}\meswo<+\infty.
\end{align}
Subsequently,
$$\int_{\Lambda_1(\Gamma)} \left|\varphi(x_1,x_2)\right|^2 \mesw<+\infty$$
almost everywhere on $\WRo$, and the function $x_1\mapsto\varphi(x_1,x_2)$ can be expanded as
\begin{equation}\label{func}
\varphi(x_1,x_2) = \sum\limits_{{\gamma^*}\in\Gamma^*}a_{\gamma^*}(x_2) e^{2\pi i\scal{x_1,{\gamma^*}}},
\end{equation}
where the series converges absolutely and uniformly on $\Lambda_1(\Gamma)$, for almost every fixed $x_2\in\WRo$. Here $a_{\gamma^*}(x_2)$ are the Fourier coefficients given by (\cite[p.44]{Igusa1972}):
	$$a_{\gamma^*}(x_2)=\frac{1}{vol(\Lambda_1(\Gamma))}\int_{\Lambda_1(\Gamma)} \varphi(x_1,x_2)e^{-2\pi i\scal{x_1,{\gamma^*}}} \mesw.$$
To prove the growth condition \eqref{condphi}, we start from \eqref{expansion0} and make use again of the Fubini theorem. This entails
\begin{align*}
\normqr{f}^2 &= \int_{\Lambda(\Gamma)} \left| e^{\frac{\nu}2 \scal{x_1,x_1}} \sum\limits_{{\gamma^*}\in\Gamma^*} a_{{\gamma^*}}(x_2)
			e^{2i\pi \scal{v_\chi+{\gamma^*},x_1}} \right|^2 e^{-\nu \scal{x,x}}\mes\\
			&= \int_{\WRo} 	e^{ -\nu \scal{x_2,x_2}} \left( \int_{\Lambda_1(\Gamma)}  \left| \sum\limits_{{\gamma^*}\in\Gamma^*} a_{{\gamma^*}}(x_2)
			e^{2i\pi \scal{v_\chi+{\gamma^*},x_1}} \right|^2 \mesw \right) \meswo	.	
\end{align*}
Next, by Parseval's identity in the Hilbert space $L^2(\Lambda_1(\Gamma);\mesw)$, we get
			\begin{align*}
			\int_{\Lambda_1(\Gamma)}  \left| \sum\limits_{{\gamma^*}\in\Gamma^*} a_{{\gamma^*}}(x_2)
			e^{2i\pi \scal{{\gamma^*},x_1}} \right|^2 \mesw
			= vol(\Lambda_1(\Gamma))	\sum\limits_{\gamma^* \in\Gamma^*} \left|a_{\gamma^*}(x_2)\right|^2,	
			\end{align*}
for almost every $x_2\in\WRo$. Therefore, the square norm $\normqr{f}^2$ reduces to
			\begin{align}
			\normqr{f}^2  
			&= vol(\Lambda_1(\Gamma)) \sum\limits_{\gamma^* \in\Gamma^*} \int_{\WRo} 	e^{ -\nu \scal{x_2,x_2}}
			\left|a_{\gamma^*}(x_2)\right|^2 \meswo	\nonumber
\\& = vol(\Lambda_1(\Gamma)) \sum\limits_{{\gamma^*}\in\Gamma^*}\| a_{\gamma^*}\|^2_{L^2(\WRo, e^{-\nu \scal{x_2 ,x_2}} \meswo)}. \label{normx1}	
			\end{align}
			This completes the proof.
		\end{proof}

To the exact statement of the main result of this section, notice that the subspace $\WRo$  is generated by some $\R$-linearly independent vectors, $\omega_{r+1}, \cdots ,\omega_d \in \V$.
Without lost of generality, we can assume that the $\omega_{r+1}, \cdots ,\omega_d$ are orthonormal with respect to the form $Q_2$; the restriction of the usual scalar product $\scal{.,.}$ on $\V$ to $\WRo$ (see for example \cite{thangavelu}).
That is
$$
Q_2(\omega_j, \omega_k)= \delta_{jk}; \quad j,k=r+1,\cdots,d.
$$ 
Identifying  $x_2=x_{r+1}\omega_{r+1}+\cdots+x_{d}\omega_{d}$ in $\WRo$ to its coordinates  $(x_{r+1},\cdots,x_{d})$ in $\R^{d-r}$. Then,
an orthogonal basis of $L^2\left(\WRo;e^{-\nu\norm{x_2}^2}\meswo\right)$ is
  \begin{equation}
  \mathbf{H}^\nu_{\mathbf{k}}(x_2)=\mathbf{H}_{\mathbf{k}}(\sqrt{\nu}x_2);\qquad\mathbf{k}\in (\Z^+)^{d-r},
  \end{equation}
where $\mathbf{H}_{\mathbf{k}}$ denotes the $(d-r)$-dimensional Hermite polynomials
		\begin{equation} \label{hermite}
		\mathbf{H}_{\mathbf{k}}(\xi)  = (-1)^{|\mathbf{k}|} \, e^{\norm{\xi}^2} \frac{\partial^{\mathbf{k}}} {\partial \xi_{r+1}^{k_{r+1}}  \cdots \partial \xi_{d}^{k_{d}} } \left( e^{-\norm{\xi}^2}\right);  
 \end{equation}
 with $\mathbf{k}= (k_{r+1},k_{r+2},\cdots,k_{d})\in(\Z^+)^{d-r}$, $|\mathbf{k}|= k_{r+1}+k_{r+2}+\cdots+k_{d}$ and $\xi=(\xi_{r+1},\xi_{r+2},\cdots,\xi_{d})\in\R^{d-r}$.
Therefore, by means of \eqref{condphi}, the Fourier coefficient $x_2\mapsto a_{\gamma^*}(x_2)$ belongs to the Hilbert space $L^2(\WRo; e^{-\nu\norm{x_2}^2} \meswo)$, and it can be expanded as
\begin{equation}\label{Coef:expansion}
a_{\gamma^*}(x_2)=\sum\limits_{\mathbf{k}\in (\Z^+)^{d-r}} a_{{\gamma^*},\mathbf{k}} \mathbf{H}^\nu_{\mathbf{k}}(x_2),
\end{equation}
for some complex numbers $a_{{\gamma^*},\mathbf{k}}$.
		Thus, we assert the following
	\begin{theorem}  \label{ThmbasisO}
A complex-valued function $f$ belongs to the Hilbert space $\LQc$ if and only if it can be expanded as
\begin{equation}\label{thm:expansion}
f(x):=f(x_1,x_2)=\sum\limits_{{\gamma^*}\in\Gamma^*,\mathbf{k}\in (\Z^+)^{d-r}}a_{{\gamma^*},\mathbf{k}} e^{\frac{\nu}2 \scal{x_1,x_1} + 2\pi i \scal{v_\chi+{\gamma^*},x_1} }
\mathbf{H}^\nu_{\mathbf{k}}(x_2)
		\end{equation}
for almost every $(x_1,x_2)\in\WR\times\WRo$, with
\begin{align*}\label{thm:norm}
			\normqr{f}^2  =vol(\Lambda_1(\Gamma)) \left(\frac{\pi}{\nu}\right)^{(d-r)/2} \sum\limits_{\gamma^*\in\Gamma^*,\mathbf{k}\in (\Z^+)^{d-r}} 2^{|\mathbf{k}|}\mathbf{k}!\left| a_{{\gamma^*},\mathbf{k}}\right|^2 <+\infty.
			\end{align*}
\end{theorem}

\begin{remark}
The series in \eqref{thm:expansion} converges in the Hilbert space $\LQc$.
\end{remark}

\begin{proof}
The result is a consequence of Theorem \ref{thm-description} and the fact that the Fourier coefficients $a_{\gamma^*}(x_2)$ are given by  \eqref{Coef:expansion}. More exactly, using the orthogonality of the Hermite polynomials in $L^2(\R^{d-r},e^{-\|\xi\|^2}d\lambda(\xi))$, we get
\begin{align*}
\|a_{\gamma^*}\|^2_{L^2(\WRo, e^{-\nu\scal{x_2,x_2}} \mess)}
&= \int_{\WRo} 	e^{ -\nu \scal{x_2,x_2}} \left|a_{\gamma^*}(x_2)\right|^2 \meswo\\
           \\& = \frac{1}{\nu^{(d-r)/2}} \sum\limits_{\mathbf{k}\in (\Z^+)^{d-r}} \left| a_{{\gamma^*},\mathbf{k}}\right|^2 \int_{\R^{d-r}} 	
           \left| \mathbf{H}_{\mathbf{k}}(\xi)\right|^2  e^{ -\norm{\xi}^2}d\lambda(\xi)
\\& =\left(\frac{\pi}{\nu}\right)^{(d-r)/2}  \sum\limits_{\mathbf{k}\in (\Z^+)^{d-r}} 2^{|\mathbf{k}|}\mathbf{k}!\left| a_{{\gamma^*},\mathbf{k}}\right|^2 .
\end{align*}	
\end{proof}

Now, we are able to give a proof of Theorem \ref{mainresult1} saying that the family of functions
		\begin{equation}\label{basis}
		e_{{\gamma^*},\mathbf{k}}(x) = e^{\frac{\nu}2 \scal{x_1,x_1}+ 2\pi i \scal{v_\chi+{\gamma^*},x_1} } \mathbf{H}^\nu_{\mathbf{k}}(x_2),
		\end{equation}
for ${\gamma^*}\in\Gamma^*$ and  $\mathbf{k}=(k_{1},  \cdots ,  k_{d-r})\in(\Z^+)^{d-r}$, constitutes an orthogonal basis of $\LQc$
with \begin{equation}\label{norm:e}			
\normqr{ e_{{\gamma^*},\mathbf{k}}}^2=vol(\Lambda_1(\Gamma))\left(\frac{\pi}{\nu}\right)^{(d-r)/2}	 2^{|\mathbf{k}|}\mathbf{k}!.
			\end{equation}

\begin{proof}[Proof of Theorem \ref{mainresult1}]	
The orthogonality of the functions $e_{{\gamma^*},\mathbf{k}}$; for $\gamma^*\in\Gamma^*$,
$\mathbf{k}=(k_{1},  \cdots ,  k_{d-r})\in(\Z^+)^{d-r}$, follows from the orthogonality of the Hermite polynomials in $L^2\left(\R^{d-r},e^{-\norm{\xi}^2}d\lambda\right)$ and the use of the following well-established identity
\begin{align}
\int_{\Lambda_1(\Gamma)}e^{2\pi i \scal{x_1,{\gamma^*}_1-{\gamma^*}_2}} \mesw \nonumber
&= vol(\Lambda_1(\Gamma))\prod_{j=1}^{r} \left(\int_0^1 e^{2\pi i t_j\scal{\omega_j,{\gamma^*}_1-{\gamma^*}_2}} dt_j\right)
=vol(\Lambda_1(\Gamma)) \, \delta_{\gamma^*_1,\gamma^*_2}  \label{int-orthog}
\end{align}
for $x_1=t_1\omega_1+ \cdots +t_{r}\omega_{r}\in \Lambda_1(\Gamma)$
with $t_j \in[0,1]$. In fact, we obtain
\begin{align*}
	\scalqr{e_{{\gamma^*},\mathbf{k}},e_{{\gamma^{*'}},\mathbf{k}'}}
	& = \left(\int_{\Lambda_1(\Gamma)}e^{2\pi i \scal{x_1,{\gamma^*}_1-{\gamma^*}_2}} \mesw\right)
   \times \left(\int_{\WRo} \mathbf{H}^\nu_{\mathbf{k}}(x_2) \overline{\mathbf{H}^\nu_{\mathbf{k'}} (x_2)} e^{-\nu \scal{x_2, x_2}} \meswo\right)
  \\& =\left(\frac{\pi}{\nu}\right)^{(d-r)/2} vol(\Lambda_1(\Gamma))2^{|\mathbf{k}|}\mathbf{k}!\delta_{{\gamma^*},{\gamma^{*'}}} \delta_{\mathbf{k},\mathbf{k}'}.
\end{align*}
To conclude, we need to prove completeness. By the uniqueness of the Fourier coefficients in the obtained expansion
\begin{equation*}
f(x_1,x_2)=\sum\limits_{{\gamma^*}\in\Gamma^*,\mathbf{k}\in (\Z^+)^{d-r}}a_{{\gamma^*},\mathbf{k}}  e^{\frac{\nu}2\scal{x_1,x_1}+2\pi i\scal{x_1,v_\chi+{\gamma^*}}}\mathbf{H}^\nu_{\mathbf{k}}(x_2)
\end{equation*}
for given $f \in \LQc$, and the completeness of the Hermite polynomials $(\mathbf{H}_{\mathbf{k}})_{\mathbf{k}}$ in the Hilbert space $L^2(\R^{d-r}; e^{-\norm{\xi}^2}d\lambda)$, we conclude that $a_{{\gamma^*},\mathbf{k}}=0$ for every ${\gamma^*}\in\Gamma^*$ and $\mathbf{k}\in(\Z^+)^{d-r}$, whenever $\scalqr{f,e_{{\gamma^*},\mathbf{k}}} =0$ for every ${\gamma^*}\in\Gamma^*$ and $\mathbf{k}\in(\Z^+)^{d-r}$.
This implies $f=0$ and the proof is completed.
		\end{proof}
		
We conclude this section by determining the completion of $\LQ$ with respect to the scalar product \eqref{IntrScal1}.
Namely, we assert

\begin{proposition}\label{prop:completion}
	The completion of $(\LQ;\normqr{.})$ coincides with the Hilbert space $\LQc$.
\end{proposition}

\begin{proof} The completion of $(\LQ;\normqr{.})$ is clearly contained in the  Hilbert space $\LQc$.
For the converse, notice that the function $f_{|_{\Lambda(\Gamma)}}$ belongs to $L^2(\Lambda(\Gamma);e^{-\nu\norm{x}^2}d\lambda)$ whenever $f\in\LQc$. Thus, there exists a sequence $(\phi_n)_n$  of $\mathcal{C}^\infty$ functions with support included in the interior of  $\Lambda(\Gamma)$. Since $\Lambda(\Gamma)$ is arbitrary, it follows that $\phi_n$ is defined on the whole $\V$ and furthermore satisfying \eqref{FctEq1}. Hence $(\phi_n)_n$ is a Cauchy sequence belonging to $(\LQ;\normqr{.})$ and converges to $f$. This shows that $f$ belongs to the completion of $\LQ$.
\end{proof}

\section{On theta Bargmann-Fock space $\Fnugr(\Vc)$}
	In this section, we review briefly some needed properties of the $(\Gamma,\chi)$-theta Bargmann-Fock space, i.e., the Hilbert $\Fnugr(\Vc)$ of the $(L^2,\Gamma,\chi)$-holomorphic theta functions on $\C^d$.  Here, we restrict ourself to the special case of $\Gamma$ being a discrete subgroup of rank $r$ in $(\V,+)$ that can be viewed as a discrete subgroup in $(\Vc,+$). The general case of $\Gamma$ being an arbitrary isotropic discrete subgroup is considered and discussed in \cite{GhInSou2016}. In fact, we regard $\Vc$ as the complexify of $\V$,  $\Vc=\V+i\V$, that we endow with the natural extension of the bilinear symmetric form $\scal{\cdot,\cdot}$ on $\V$, to wit
	 \begin{equation}
	\scal{z,w} = u_1 v_1+ u_2 v_2 + \cdots + u_d v_d ,
	\end{equation}
for $z=(u_1,u_2,\cdots,u_d)$ and $w=(v_1,v_2,\cdots,v_d)$ in $\Vc$,  as well as the standard hermitian scalar product
	 \begin{equation}
	H(z,w) = \scal{z,\overline{w} }.
	\end{equation}
Accordingly, the decomposition $\V= \WR + \WRo$ can be extended naturally to
		 \begin{equation}
		\label{splitcomp}
		\Vc=\WC\oplus\WCo,
		\end{equation}
where $\WC$ and $\WCo$ are the complex subspaces $\WC=\WR+i\WR$ and $\WCo=\WRo+i\WRo$.
Hence, the considered discrete rank $r$ subgroup $\Gamma$ is an isotropic subgroup of $(\Vc,+)$,
in the sense that $\Im m(H(\gamma, \gamma'))=0$ for all $\gamma,\gamma'\in\Gamma$.

Within the above notations,  $\Fnugr(\Vc)$ is the space of all holomorphic functions on $\C^d$ displaying the functional equation
	\begin{equation}\label{FctEq3}
	f(z+\gamma) = \chi(\gamma)  e^{ \nu H(z+\frac{\gamma}2, \gamma) } f(z);\quad z\in \Vc, \, \gamma\in\Gamma,
	\end{equation}
and such that the square norm
	\begin{equation}\label{IntrNorm2}
	\normnu{f}^2 := \int_{\tilde{\Lambda}(\Gamma)} |f(z)|^2 e^{-\nu H(z, z)} \mesc,
	\end{equation}
is finite. Here $\tilde{\Lambda}(\Gamma)$ is a fundamental domain of $\Gamma$ in $\Vc=\R^{2d}$. 
The associated hermitian inner scaler product is
	\begin{equation}\label{IntrScal2}
	\scalnue{f,g} := \int_{\tilde{\Lambda}(\Gamma)} f(z) \overline{g(z)}  e^{-\nu H(z,z)} \mesc.
	\end{equation}

The needed properties related to $\Fnugr(\Vc)$ are summarized in the following

	\begin{theorem}\label{descriptionfock}
	 The space $\Fnugr(\Vc)$ is nontrivial if and only if $\chi$ is a character. In this case, the $(\Gamma, \chi)$-theta Bargmann-Fock space $\Fnugr(\Vc)$ is a reproducing kernel Hilbert space. Moreover, the set of functions
	  \begin{equation}\label{basisthetafock}
		\varphi_{\gamma^*,\mathbf{k}}(z_1,z_2)=e^{ \frac{\nu}{2}\scal{z_1,z_1}+2\pi i\scal{z_1,\gamma^*+v_\chi}}(z_2)^{\mathbf{k}};\qquad z_1\in\WC,\, z_2\in\WCo,
		\end{equation}
		for varying $\gamma^*\in\Gamma^*$ and multi-index  $\mathbf{k}\in (\Z^+)^{d-r}$,
		constitutes an orthogonal basis of $\Fnugr(\Vc)$ with
		\begin{align}\label{norm-basiss}
		\normnu{\varphi_{\gamma^*,\mathbf{k}}}^2 =\left(\frac{\pi}{\nu}\right)^{d-\frac{r}{2}} \frac{vol(\Lambda_1(\Gamma))}{2^{r/2}}\frac{k!}{\nu^{|\mathbf{k}|}} e^{2\frac{\pi^2}{\nu}\scal{\gamma^*+v_\chi,\gamma^*+v_\chi}}.
		\end{align}
	\end{theorem}

	The proofs of the first statements are quite similar to the one provided in \cite{GhInSou2016}.
To prove the last one, we make use of the following Lemma.
		\begin{lemma}[\cite
{Bargmann1967}] \label{lem:FourierCoef}
A $\Gamma$-periodic holomorphic function $h$ on $\WC$ can be expanded as follows
			$$h(w) =  \sum\limits_{\gamma^*\in\Gamma^*} b_{\gamma^*}(w) e^{ 2\pi i\scal{w,\gamma^*}} , $$
			where the series converges absolutely and uniformly on every compact subset of $\WC$. Furthermore, the coefficients $b_{\gamma^*}(w)$ are given by
		\begin{align}\label{FourierCoef}
        b_{\gamma^*}(w)=\frac{1}{vol(\Lambda_1(\Gamma))}\int_{\Lambda_1(\Gamma)} h(w) e^{-2\pi i\scal{\gamma^*,w}} d\lambda(x_1);\quad w=x_1+iy_1,
        \end{align}
			and are independents of $\Im m(w)$.
		\end{lemma}

			\begin{proof}[Proofs of \eqref{basisthetafock} and \eqref{norm-basiss}]
	Associated to a given holomorphic function $f$ displaying \eqref{FctEq3}, we consider
				\begin{equation}\label{expans}
					h(z):=h(z_1,z_2)= e^{ -\frac{\nu}{2}\scal{z_1,z_1} -2\pi i\scal{z_1,v_\chi}} f(z_1,z_2).
				\end{equation}
	Hence, we can easily check that $h$ is a holomorphic $\Gamma$-periodic function, and therefore the function $z_1\mapsto h_{z_2}(z_1):=h(z_1,z_2)$ can be expanded as
				\begin{equation}\label{func}
					h_{z_2}(z_1) =  \sum\limits_{\gamma^*\in\Gamma^*} b_{\gamma^*}(z_2) e^{ 2\pi i\scal{z_1,\gamma^*}},
				\end{equation}
for fixed $z_2$ in $\WCo$, according to Lemma \ref{lem:FourierCoef}. 
The Fourier coefficients $b_{\gamma^*}(z_2)$ are given through \eqref{FourierCoef} with $h(z_1,z_2)=h_{z_2}(z_1)$
and they are holomorphic on $\WCo$. Thus, they can be written as
				$$b_{\gamma^*}(z_2)=\sum\limits_{\mathbf{k}\in(\Z^+)^{d-r}}b_{\gamma^*,\mathbf{k}}z_2^{\mathbf{k}},$$
for some complex numbers $b_{\gamma^*,\mathbf{k}}$. Above, $z_2$ is identified with their coordinates in $\WCo$ and the functions $\varphi_{\gamma^*,\mathbf{k}}$ are generators of $\Fnugr(\Vc)$. For the orthogonality,
 we use Fubini's theorem to get
			\begin{align*}
			\scalnue{\varphi_{\gamma^*,\mathbf{k}}, \varphi_{{\gamma^*}',\mathbf{k}'}}
			&= \left(\int_{\Lambda_1(\Gamma)\times \WR} e^{\frac{\nu}{2} (\scal{z_1,z_1} + \scal{\overline{z_1},\overline{z_1}} -2 H(z_1,z_1) )
				+ 2i\pi \scal{z_1 - \overline{z_1},v_\chi} + 2i\pi( \scal{\gamma^*,z_1}  - \scal{{\gamma^*}', \overline{z_1} }) }   \meswc\right)
			\\ & \qquad \times   \left(\int_{\C^{d-r}}  z_{2}^{\mathbf{k}} \overline{z_{2}}^{\mathbf{k}'} e^{- \nu |z_2|^2_{\C^{d-r}} }\meswoc \right).
			\end{align*}
In the second integral in the right hand-side of the last identity, we recognize the scalar product of the monomials in the classical Bargmann-Fock space on $\C^{d-r}$, to wit
			$$  \int_{\C^{d-r}}  z_{2}^{\mathbf{k}} \overline{z_{2}}^{\mathbf{k}'} e^{-\nu |z_2|^2_{\C^{d-r}} }\meswoc
			=  \left(\frac{\pi}{\nu}\right)^{d-r}  \frac{k!}{\nu^{|\mathbf{k}|}} \delta_{\mathbf{k},\mathbf{k}'} .$$
	For $z_1= x_1+iy_1$ with $x_1,y_1\in \WR$, we have $\scal{\gamma^*,z_1}  - \scal{{\gamma^*}', \overline{z_1} } =\scal{\gamma^*-{\gamma^*}',x_1}  + i\scal{{\gamma^*}'+{\gamma^*}, y_1 } $.
	Note also that since $H(z_1,z_1) = \scal{z_1,\overline{z}}$, it follows
			$\scal{z_1,z_1} + \scal{\overline{z_1},\overline{z_1}} -2 H(z_1,z_1) = - 4 \scal{y_1, y_1}.$
	Hence, we get
			\begin{align*}
			\scalnue{\varphi_{\gamma^*,\mathbf{k}}, \varphi_{{\gamma^*}',\mathbf{k}'}}
& =  \left(\frac{\pi}{\nu}\right)^{d-r}  \frac{k!}{\nu^{|\mathbf{k}|}}
              \left(\int_{\Lambda_1(\Gamma)} e^{2i\pi\scal{\gamma^*-{\gamma^*}',x_1} }\mesw \right)
               \nonumber  \\ &\qquad \qquad \qquad \qquad \qquad \times
                \left(\int_{\WR} e^{ - 2\scal{y_1, y_1}   - 2\pi\scal{{\gamma^*}'+{\gamma^*}+2v_\chi, y_1 } }  d\lambda(y_1)\right)  \delta_{\mathbf{k},\mathbf{k}'}    \nonumber
			\\ &=  \left(vol(\Lambda_1(\Gamma))\right)^2 \left(\frac{\pi}{\nu}\right)^{d-r}  \frac{k!}{\nu^{|\mathbf{k}|}} \left(\int_{\R^r} e^{ - 2 y_1Gy_1  - 4\pi (\beta_\chi+m)Gy_1}   d\lambda(y_1)\right) \delta_{\gamma^*,{\gamma^*}'} \delta_{\mathbf{k},\mathbf{k}'}.
			\end{align*}
		To conclude, we need to the following explicit expression for the gaussian integral
		
 \begin{lemma}[\cite
 {Folland2004}]
 \label{gaussIntegral} Let $a>0$, $b\in \C^s$ and $A$ a symmetric $s\times s$ matrix whose $\Re e
		 	(A)$ is positive definite. Then,
		 	\begin{align}\label{GImatrix}
		 	\int_{\R^s} e^{-a y A y + by} dy = \left(\frac{1}{\sqrt{\det A}}\right)\left(\dfrac{\pi}{a}\right)^{s/2}  e^{\frac{1}{4 a} b A^{-1} b }.
		 	\end{align}	
 \end{lemma}
Hence by means of \eqref{GImatrix}, we obtain
					\begin{align}
					\scalnue{\varphi_{\gamma^*,\mathbf{k}}, \varphi_{{\gamma^*}',\mathbf{k}'}} =
                    \left(\frac{\pi}{\nu}\right)^{d-\frac{r}{2}} \frac{vol(\Lambda_1(\Gamma))}{2^{r/2}}\frac{k!}{\nu^{|\mathbf{k}|}} e^{2\frac{\pi^2}{\nu}\scal{\gamma^*+v_\chi,\gamma^*+v_\chi}} \delta_{\gamma^*,{\gamma^*}'}\delta_{\mathbf{k},\mathbf{k}'}.
					\end{align}
					The completeness can be obtained by proceeding in a similar way as in \cite{GhInSou2016}.
			\end{proof}

	\section{Proofs of Theorems \ref{mainresult2} and \ref{mainresult3}: the theta-Segal-Bargmann transform  }

In this section we look for the action of the so-called Segal-Bargmann transform $\mathcal{B}$ on $\LQc$. Recall that for given $\varphi:\V \longrightarrow \C$, we have \cite{Bargmann1961}
		\begin{align}\label{bargmann}
		[\mathcal{B}\varphi](z) =\left(\frac{\nu}{\pi}\right)^{\frac{3d}{4}}
		\int_{\V} e^{\sqrt{2}\nu \scal{z,x} -\frac{\nu}{2}\scal{z,z}} \varphi(x)e^{-\nu \norm{x}^2}\mes; \quad z\in \Vc=\R^{2d},
		\end{align}
provided that the integral exists. Then, it is a well-known fact that this transform maps isometrically the classical space $L^2(\V; e^{-\nu \norm{x}^2}d\lambda)$, of all $e^{-\nu \norm{x}^2}d\lambda$-square integrable functions on $\V$, onto the classical Bargmann-Fock space $\mathcal{F}^{2,\nu}(\Vc)$. Its kernel function turns out to be the exponential generating function of the Hermite polynomials.
Notice that the space $L^2(\V; e^{-\nu \norm{x}^2}d\lambda)$ (resp. $\mathcal{F}^{2,\nu}(\Vc)$) corresponds to $\LQc$ (resp. $\Fnugr(\Vc)$) with $\Gamma$ and $\chi$ are trivial, $\Gamma=\{0\}$ and $\chi=1$.
In order to characterize $\mathcal{B}(\LQc)$ for arbitrary rank $r$ discrete subgroup $\Gamma$ and character $\chi$, we begin with the following
	
\begin{proposition}
 The integral operator $\mathcal{B}$ is well defined on $\LQc$ and the integral defining $\mathcal{B}\varphi$; for $\varphi \in \LQc$, converges uniformly on compact sets of $\Vc$.
	\end{proposition}

	\begin{proof} Starting from the fact $\V=\bigcup\limits_{\gamma\in\Gamma} ( \gamma +\Lambda(\Gamma))$, we can rewrite the Segal-Bargmann transform  as
		\begin{align*}
        [\mathcal{B}\varphi](z)& =\left(\frac{\nu}{\pi}\right)^{\frac{3d}{4}}e^{\frac{\nu}{2}\scal{z,z}}
		\int_{\V} e^{-\frac{\nu}{2}\scal{x-\sqrt{2}z,x-\sqrt{2}z}-\frac{\nu}{2}\scal{x,x}} \varphi(x) \mes
        \\&=\left(\frac{\nu}{\pi}\right)^{\frac{3d}{4}}e^{\frac{\nu}{2}\scal{z,z}}
		\sum_{\gamma\in\Gamma}\int_{\gamma+\Lambda(\Gamma)} e^{-\frac{\nu}{2}\scal{x-\sqrt{2}z,x-\sqrt{2}z}-\frac{\nu}{2}\scal{x,x}}  \varphi(x) \mes.
		\end{align*}
		Replacing $x$ by $x+\gamma$ in the above integral and using \eqref{FctEq2} satisfied by $\varphi$, we get
		\begin{align} \label{bseries1}
		[\mathcal{B}\varphi](z)&= 
	    \left(\frac{\nu}{\pi}\right)^{\frac{3d}{4}} e^{\frac{\nu}{2}\scal{z,z}}
         \\& \times \left( \sum_{\gamma\in\Gamma}\chi(\gamma)e^{-\frac{\nu}{2}\scal{\gamma,\gamma}+\sqrt{2}\nu \scal{z,\gamma}}\int_{\Lambda(\Gamma)}  e^{-\frac{\nu}{2}\scal{x-\sqrt{2}z,x-\sqrt{2}z}-\nu \scal{\gamma, x}-\frac{\nu}{2}\scal{x,x}}\varphi(x)\mes\right) \nonumber
		\end{align}
and therefore
			\begin{align} \label{bseries}
			\big|[\mathcal{B}\varphi](z)\big|\leq \left(\frac{\nu}{\pi}\right)^{\frac{3d}{4}}
                 \sum_{\gamma\in\Gamma}\big|e^{\frac{\nu}{2}\scal{z,z}-\frac{\nu}{4}\scal{\gamma,\gamma}+\sqrt{2}\nu \scal{z,\gamma}}  I_{\nu,\Gamma}(\varphi)(z_1,z_2)\big|,
			\end{align}
since $-\scal{\gamma,x}=-\scal{\gamma,x_1}\leq \scal{x_1,x_1}+\frac{1}{4}\scal{\gamma,\gamma}$. The quantity $I_{\nu,\Gamma}(\varphi)(z_1,z_2) $ stands for
 $$ I_{\nu,\Gamma}(\varphi)(z_1,z_2)= \int_{\Lambda_1(\Gamma)\times\WRo}  e^{-\frac{\nu}{2}\scal{x-\sqrt{2}z,x-\sqrt{2}z}+\frac{\nu}{2}\scal{x_1,x_1}-\frac{\nu}{2}\scal{x_2,x_2}}\varphi(x)\mesw\meswo.$$
The Cauchy-Schwarz inequality combined with the Fubini theorem yields the following estimate
		\begin{align*}
\left|I_{\nu,\Gamma}(\varphi)(z_1,z_2)\right|^2
	 &\leq \normqr{\varphi}^2\left(\int_{\Lambda_1(\Gamma)}\left| e^{-\frac{\nu }{2}\scal{x_1-\sqrt{2}z_1,x_1-\sqrt{2}z_1}+\frac{\nu}{2}\scal{x_1,x_1}}\right|^2\mesw\right)
\\ & \qquad \times \left(\int_{\WRo}\left|e^{-\frac{\nu}{2}\scal{x_2-\sqrt{2}z_2,x_2-\sqrt{2}z_2}-\frac{\nu}{2}\scal{x_2,x_2}}\right|^2\meswo\right).
		\end{align*}
The first integral involved in the right hand-side of the previous inequality is clearly finite for  $\Lambda_1(\Gamma)$ being compact.
 The second one can be shown to be finite using Lemma \ref{gaussIntegral}. Moreover, for every $z$ in any compact set $K\subset\Vc$, we have $\Re(\scal{\gamma,z})\leq \frac{1}{8}\scal{\gamma,\gamma}$ for $\gamma \in \Gamma$ outside certain disc $D(0,R)$; $R>0$. Thus, there exists a constant $c_K$ such that
		\begin{equation}\label{estimateba}
		\big|[\mathcal{B} \varphi](z)\big|\leq c_K\sum_{\gamma\in\Gamma} e^{-\frac{\nu}{8}\scal{\gamma,\gamma}}.
		\end{equation}
This means that the quantity $\mathcal{B} \varphi$ is well defined for every $\varphi\in\LQc$, since the series in the right hand-side of \eqref{bseries1} converges.
	\end{proof}

A part of the proof of Theorem \eqref{mainresult2} is contained in the following proposition showing the image of the $(L^2,\Gamma,\chi)$-likewise theta functions on $\V$, by the $\mathcal{B}$, are the  $(\widetilde{\Gamma},\widetilde{\chi})$-holomorphic theta functions on $\Vc$, where $\widetilde{\Gamma}$ is the scaled discrete subgroup $\widetilde{\Gamma}:=\Gamma/{\sqrt{2}}$ and $\widetilde{\chi}$ is the $\widetilde{\Gamma}$-character defined by $$\widetilde{\chi}(\widetilde{\gamma})=e^{2i\pi\scal{\widetilde{\gamma},\sqrt{2}v_\chi}} .$$

	\begin{proposition}\label{PBT} For every $\varphi\in \LQc$, the function $z\mapsto[\mathcal{B} \varphi](z)$ is holomorphic on $\Vc$ and satisfies the functional equation
		\begin{equation}\label{FctEqBargm}
		[\mathcal{B} \varphi] (z+\widetilde{\gamma})=\widetilde{\chi}(\widetilde{\gamma}) e^{\nu H(z+\frac{\widetilde{\gamma}}{2},\widetilde{\gamma})} [\mathcal{B} \varphi] (z)
		\end{equation}
		for every $z$ in $\Vc$ and every $\widetilde{\gamma}$ in $\widetilde{\Gamma}$.
	\end{proposition}
	
	\begin{proof}
For given $\varphi\in \LQc$, the function $z \mapsto  e^{\sqrt{2}\nu \scal{z,x} -\frac{\nu}{2}\scal{z,z}- \nu \|x\|^2} \varphi(x)$
involved in the integrand of $\mathcal{B} \varphi$ is clearly holomorphic in $z$ for every fixed $x\in \V$.
 Thus, by the uniform convergence of the integral in $\mathcal{B} \varphi$ on compact subsets of $\Vc$ (Proposition \ref{PBT}), it follows that
 $[\mathcal{B} \varphi] (z)$ is holomorphic on $\Vc$. Moreover, direct computation infers
\begin{align*}
		[\mathcal{B} \varphi] (z+\frac{1}{\sqrt{2}}\gamma)
		&=\left(\frac{\nu}{\pi}\right)^{\frac{3d}{4}}e^{\frac{\nu}{2}\scal{z,z}+\nu \scal{z+\frac{\widetilde{\gamma}}{2},\widetilde{\gamma}}}
		\int_{\V} e^{-\frac{\nu }{2}\scal{x-\sqrt{2}z-\gamma,x-\sqrt{2}z-\gamma}-\frac{\nu}{2}\|x\|^2} \varphi(x) \mes.
		\end{align*}
Now, making use of the change $y=x-\gamma$ as well as the fact that $\varphi$ satisfies the functional equation \eqref{FctEq2}, we obtain
		\begin{align*}
	[\mathcal{B} \varphi] (z+\widetilde{\gamma})= \chi(\gamma)e^{\nu \scal{z+\frac{\tilde{\gamma}}{2},\tilde{\gamma}}}[\mathcal{B} \varphi] (z)
= \widetilde{\chi}(\widetilde{\gamma}) e^{\nu H(z+\frac{\widetilde{\gamma}}{2},\widetilde{\gamma})} [\mathcal{B} \varphi] (z),
		\end{align*}
since $\scal{z+\frac{\tilde{\gamma}}{2},\tilde{\gamma}}=H(z+\frac{\tilde{\gamma}}{2},\tilde{\gamma})$ for $\overline{\tilde{\gamma}}=\tilde{\gamma} $
		and  $\chi(\gamma)=e^{2i\pi\scal{\gamma,v\chi}}=e^{2i\pi\scal{\frac{\gamma}{\sqrt2},\sqrt{2}v\chi}}=\tilde{\chi}(\tilde{\gamma})$
		by Lemma \ref{Lem:condidatFct}.
\end{proof}
	
The proof of Theorem \ref{mainresult2}, i.e., 	$\mathcal{B} (\LQc)= \Fnugrt(\Vc)$ isometrically, reduces further to show that the Segal-Bargmann transform $\mathcal{B}$ maps an orthonormal basis $e_{\gamma^*,\mathbf{k}}$ of the Hilbert space $\LQc$ to the an orthonormal one of the Hilbert space $\Fnugrt(\Vc)$.

\begin{proof}[Proof of Theorem \ref{mainresult2}]
Notice first that the action of $\mathcal{B}$ on the basis $e_{\gamma^*,\mathbf{k}}$ of the Hilbert space $\LQc$ is given by
		\begin{align}\label{evalu}
		[\mathcal{B}e_{\gamma^*,\mathbf{k}}](z)
		& =\left(\frac{\nu}{\pi}\right)^{\frac{3d}{4}} e^{-\frac{\nu}{2}\scal{z_1,z_1}}\left(\int_{\WR}
		e^{\sqrt{2}\nu \scal{x_1,z_1}+2i\pi \scal{x_1,\gamma^*+v_\chi}-\frac{\nu}{2}\scal{x_1,x_1}} \mesw \right)
\\ &\qquad \times \left(\int_{\WRo}  e^{-\frac{\nu}{2}\scal{z_2,z_2}+\sqrt{2}\nu \scal{x_2,z_2}} \mathbf{H}^\nu_{\mathbf{k}}(x_2)e^{-\nu\scal{x_2,x_2}}\meswo\right).\nonumber
		\end{align}
In the right hand-side of the second integral occurring in \eqref{evalu}, we recognize the action of the Segal-Bargmann transform of the Hermite polynomial on $\R^{d-r}$  given by
		\begin{equation}\label{e1}
		\left(\frac{\nu}{\pi}\right)^{\frac{3(d-r)}{4}}\int_{\WRo}  e^{-\frac{\nu}{2}\scal{z_2,z_2}+\sqrt{2}\nu \scal{x_2,z_2}} \mathbf{H}^\nu_{\mathbf{k}}(x_2)e^{-\nu\scal{x_2,x_2}}\meswo=\left(\frac{\nu}{\pi}\right)^{\frac{d-r}{4}}(2\nu)^{|\mathbf{k}|/2}z_2^{\mathbf{k}}.
		\end{equation}
The first integral in the right hand-side of \eqref{evalu} can be handled by applying Lemma \ref{gaussIntegral}. Indeed, if $G:=\left(\scal{\omega_j,\omega_k}\right)_{j,k=1}^r$ denotes the Gram-Schmidt matrix of $Q_1=\scal{.,.}$ on $\WR$ with respect to the basis $\omega_1,\cdots,\omega_r$, and by $x_1\in\R^r$, $z_1\in\C^r$, $m^*\in\R^r$ and $\beta_\chi\in\R^r$ the coordinates of $x_1\in\WR$, $z_1\in\WC$, $\gamma^*$ and $v_\chi$, respectively, we get
\begin{align}
	\int_{\WR} 	& e^{\sqrt{2}\nu \scal{x_1,z_1}+2i\pi \scal{x_1,\gamma^*+v_\chi}-\frac{\nu}{2}\scal{x_1,x_1}} \mesw \nonumber
\\& \qquad\qquad =vol(\Lambda_1(\Gamma))\int_{\R^r} e^{-\frac{\nu}{2} x_1 G x_1 + x_1 G \left(\sqrt{2} \nu z_1+2i\pi(m^*+\beta_\chi)\right)} \mesw \nonumber
\\ & \qquad\qquad =\left(\frac{2\pi}{\nu}\right)^{\frac{r}{2}}e^{\nu \scal{z_1,z_1}+2i\pi\sqrt{2} \scal{z_1,\gamma^*+v_\chi}-2\frac{\pi^2}{\nu}\scal{\gamma^*+v_\chi,\gamma^*+v_\chi}}\nonumber
\\&\qquad\qquad \stackrel{\mbox{Lemma }\ref{gaussIntegral}}{=}
\left(\frac{2\pi}{\nu}\right)^{\frac{r}{2}}e^{\nu \scal{z_1,z_1}+2i\pi \scal{z_1,\widetilde{\gamma}^*+\widetilde{v_\chi}}-\frac{\pi^2}{\nu}\scal{\widetilde{\gamma}^*+\widetilde{v_\chi},\widetilde{\gamma}^*+\widetilde{v_\chi}}}.
\label{e2}
	\end{align}
Finally, by taking into account the fact that $\widetilde{\Gamma}^*=\sqrt{2}\Gamma^*$ and inserting \eqref{e2} and \eqref{e1} in \eqref{evalu}, one obtains
$ [\mathcal{B}e_{\gamma^*,\mathbf{k}}](z) =  \varphi_{\widetilde{\gamma}^*,\mathbf{k}}.$
This completes the proof, since $ \varphi_{\widetilde{\gamma}^*,\mathbf{k}} $, for varying $\widetilde{\gamma}^*\in \widetilde{\Gamma}^*$ and $\mathbf{k}\in (\Z^+)^{d-r}$, constitute a complete orthogonal system of $\Fnugrt(\Vc)$.
	\end{proof}

Below, we give a proof of Theorem \ref{mainresult3} saying that the Segal-Bargmann transform on $\LQc$ can be realized as
		\begin{equation}\label{TransfBargm2}
			[\mathcal{B}\varphi](z)= \left(\frac{\nu}{\pi}\right)^{\frac{3d}{4}} \int_{\Lambda(\Gamma)}
           e^{\sqrt{2}\nu \scal{z,x}-\frac{\nu}{2}\scal{z,z}} \Theta_{0,G\beta_\chi}\left(\frac{i\nu}{2\pi}G(x_1-\sqrt{2}z_1)\bigg|
          \frac{i\nu}{2\pi}G\right) \varphi(x)e^{-\nu \|x\|^2} \mes, 
		\end{equation}
where $\beta_\chi=(\beta_1,\cdots,\beta_r)\in\R^r$ are the coordinates of $v_\chi$ in the basis $\omega_1,\cdots,\omega_r$.

\begin{proof}[Proof of Theorem \ref{mainresult3}]
	Starting from \eqref{bseries1}, we need to give a closed expression of the sum in the right hand-side in terms of the modified theta function. In fact, for $\gamma\in\Gamma$ with the coordinates $m\in\Z^r$, we have
		\begin{align*}	\sum_{\gamma\in\Gamma}\chi(\gamma)e^{-\frac{\nu}{2}\scal{\gamma,\gamma}+\nu\scal{\sqrt{2}z-x,\gamma}}
			&=\sum_{\gamma\in\Gamma}e^{-\frac{\nu}{2}\scal{\gamma,\gamma}+\nu\scal{\sqrt{2}z_1-x_1+\frac{2i\pi}{\nu}v_\chi,\gamma}}\\
			&=\sum_{m\in\Z^r}e^{-\frac{\nu}{2}mGm+\nu mG(\sqrt{2}z_1-x_1+\frac{2i\pi}{\nu}\beta_\chi)}\\&=
			\Theta_{0,G\beta_\chi}\left(\frac{i\nu}{2\pi}G(x_1-\sqrt{2}z_1)\bigg| \frac{i\nu}{2\pi}G\right).
		\end{align*}
The vectors $x_1$ and $z_1$ are identified with their coordinates relatively to the basis $\omega_1,\cdots,\omega_r$ 
as mentioned in proof of Lemma \ref{Lem:condidatFct}.
	\end{proof}

We conclude this paper by noting that the kernel function
\begin{equation}
			A^\nu_{\Gamma,\chi}(z;x) :=\left(\frac{\nu}{\pi}\right)^{\frac{3d}{4}}e^{\sqrt{2}\nu \scal{z,x}-\frac{\nu}{2}\scal{z,z}}
                                      \Theta_{0,G\beta_\chi}\left(\frac{i\nu}{2\pi}G(x_1-\sqrt{2}z_1)\bigg| \frac{i\nu}{2\pi}G\right),
		\end{equation}
is in fact the bilateral generating function corresponding to the orthogonal basis $e_{\gamma^*,\mathbf{k}}$ of $\LQc$ and $\varphi_{\sqrt{2}\gamma^*,\mathbf{k}}$ of $\Fnugrt(\Vc)$. Namely, we assert the following

\begin{proposition} We have
         		\begin{equation}\label{GenFct}
A^\nu_{\Gamma,\chi}(z;x) = \sum_{\gamma^*\in\Gamma^*,\mathbf{k}\in(\Z^+)^{d-r}}
\frac{\varphi_{\sqrt{2}\gamma^*,\mathbf{k}}(z)}{\left\Vert \varphi_{\sqrt{2}\gamma^*,\mathbf{k}} \right\Vert_{\widetilde{\Gamma},H}}
\frac{\overline{e_{\gamma^*,\mathbf{k}}(x)}}{\normqr{e_{\gamma^*,\mathbf{k}}}} .
		\end{equation}
		\end{proposition}

		\begin{proof}
		Using the explicit expressions of $\varphi_{\sqrt{2}\gamma^*,\mathbf{k}}$ given by \eqref{basisthetafock} and of $e_{\gamma^*,\mathbf{k}}$ given by \eqref{thm:expansion} as well as their norms given respectively through \eqref{norm-basiss} and \eqref{norm:e}, we can rewrite the right hand-side of
\eqref{GenFct} as
		\begin{align*}
		&	\frac{2^{r/2}} {vol(\Lambda_1(\Gamma))}   \left(\frac{\nu}{\pi}\right)^{\frac{3d}{4}-\frac{r}{2}}e^{ \frac{\nu}{2}(\scal{z_1,z_1}+\scal{x_1,x_1})}
  \\& \times \sum_{\gamma^*\in\Gamma^*,\mathbf{k}\in(\Z^+)^{d-r}}e^{-\frac{\pi^2}{\nu}\scal{\sqrt{2}\gamma^*+\sqrt{2}v_\chi,\sqrt{2}\gamma^*+\sqrt{2}v_\chi}+2\pi i\scal{\sqrt{2}z_1-x_1,\gamma^*+v_\chi}}\frac{\left(\sqrt{\frac{\nu}{2}}z_2\right)^{\mathbf{k}}\mathbf{H}^\nu_{\mathbf{k}}(x_2)}			 {\mathbf{k}!}.\end{align*}
Using the generating formula for Hermite polynomials (see for example \cite[p. 60]{Lebedev1972}), we get
			 \begin{equation*}
			 \sum_{\mathbf{k}\in(\Z^+)^{d-r}}\frac{\left(\sqrt{\frac{\nu}{2}}z_2\right)^{\mathbf{k}}\mathbf{H}^\nu_{\mathbf{k}}(x_2)} {\mathbf{k}!}=e^{-\frac{\nu}{2}\scal{z_2,z_2}+\sqrt{2}\nu \scal{x_2,z_2}}.
			 \end{equation*}
			 Now, note that the dual lattice $\Gamma^*$ can be identified with $G^{-1}\Gamma$, and hence
			 \begin{align*}
			 \sum_{\gamma^*\in\Gamma^*}
                      & e^{-\frac{\pi^2}{\nu}\scal{\sqrt{2}\gamma^*+\sqrt{2}v_\chi,\sqrt{2}\gamma^*+\sqrt{2}v_\chi}+2\pi i\scal{\sqrt{2}z_1-x_1,\gamma^*+v_\chi}}
                      \\&=e^{-\frac{2\pi^2}{\nu}\scal{v_\chi,v_\chi}+2i\pi\scal{\sqrt{2}z_1-x_1,v_\chi}}
                      \sum_{m\in\Z^r}e^{-\frac{2\pi^2}{\nu}mG^{-1}m+2i\pi(\sqrt{2}z_1-x_1+\frac{2i\pi}{\nu}\beta_\chi)m}
                      \\&=e^{-\frac{2\pi^2}{\nu}\scal{v_\chi,v_\chi}+2i\pi\scal{\sqrt{2}z_1-x_1,v_\chi}}
			 \Theta_{0,0}\left(\sqrt{2}z_1-x_1+\frac{2i\pi}{\nu}\beta_\chi\bigg|- \left(\frac{i\nu}{2\pi}G\right)^{-1}\right).
			 \end{align*}
Thus, in view of the well-known identity satisfied by the theta function \cite[p. 195]{Mumford83}
		\begin{equation}
			\Theta_{0,0}\left(\Omega^{-1} z\big|-\Omega^{-1}\right)=\sqrt{\det\left(- i \Omega\right)}e^{i\pi z\Omega^{-1} z}\Theta_{0,0}\left(z\big|\Omega\right),
		\end{equation}
		with $\Omega=\frac{i\nu}{2\pi}G$, one can see that the left hand-side in \eqref{GenFct} reduces further to
		\begin{align}
			\left(\frac{\nu}{\pi}\right)^{\frac{3d}{4}}e^{\sqrt{2}\nu \scal{z,x}-\frac{\nu}{2}\scal{z,z}}
           \Theta_{0,G\beta_\chi}\left(\frac{i\nu}{2\pi}G(x_1-\sqrt{2}z_1)\bigg| \frac{i\nu}{2\pi}G\right)\nonumber=A^\nu_{\Gamma,\chi}(z;x).
		\end{align}
	\end{proof}

			\noindent{\bf Acknowledgement:}
			The assistance of the members of the seminars "Partial differential equations and spectral geometry" is gratefully acknowledged.
			A. Ghanmi, A. Intissar and Z. Mouhcine are partially supported by the Hassan II Academy of Sciences and Technology. M. Ziyat is partially supported by the CNRST grant 56UM5R2015, Morocco.

\end{document}